\documentclass[11pt]{article}

\usepackage[a4paper,margin=1in]{geometry}

\usepackage[T1]{fontenc}
\usepackage[utf8]{inputenc}

\usepackage{amsmath,amssymb,amsfonts,mathrsfs}
\usepackage{amsthm}
\usepackage{bm}
\usepackage[numbers,sort&compress]{natbib}

\newtheorem{theorem}{Theorem}[section]

\newtheorem{proposition}[theorem]{Proposition}
\newtheorem{corollary}[theorem]{Corollary}
\theoremstyle{definition}
\newtheorem{definition}[theorem]{Definition}
\theoremstyle{remark}
\newtheorem{remark}[theorem]{Remark}

\usepackage{graphicx}
\usepackage{subcaption}

\usepackage{array,booktabs,makecell}
\newcolumntype{P}[1]{>{\centering\arraybackslash}p{#1}}

\usepackage[colorlinks=true,
            linkcolor=blue,
            citecolor=blue,
            urlcolor=blue]{hyperref}

\begin{document}

\title{Global Attractors for Dissipative Flows on Degenerate Constraint Manifolds}

\author{Prasanta Sahoo\\
\small Midnapore College (Autonomous), Midnapore, West Bengal 721101, India\\
\small Email: \texttt{prasantmath123@yahoo.com}
}

\maketitle


\noindent\textbf{Subject Classification numbers:} 37L30, 37D10, 37J25, 53C50

\begin{abstract}

A class of dissipative dynamical systems evolving on smooth constraint hypersurfaces endowed with degenerate induced bilinear forms is studied. The intrinsic evolution is generated by constraint--preserving vector fields on manifolds whose tangent bundles admit a nontrivial null distribution associated with the degeneracy of the induced structure. In this indefinite setting, the absence of coercive Lyapunov functionals prevents the direct application of classical attractor theory developed for Riemannian phase spaces. Dissipation is instead characterized relative to functionals that are compatible with the null distribution and exhibit decay exclusively in directions transverse to the associated foliation. Under suitable involutivity and regularity assumptions on the null distribution, all bounded trajectories are shown to be asymptotically confined to invariant leaves of the corresponding foliation. Asymptotic compactness of the intrinsic evolution is then established without coercivity by reducing the dynamics to a projected semiflow on the quotient manifold determined by the characteristic distribution. In the presence of a bounded absorbing set and continuous semiflow structure, the intrinsic evolution admits a compact global attractor saturated by the null leaves, whose effective asymptotic dynamics are governed by a compact invariant subset of the reduced phase space. Furthermore, when the compatible functional satisfies a Morse--type nondegeneracy condition in directions transverse to the null distribution and the induced transversal linearization admits no center spectrum, the projected semiflow possesses no neutral directions. The resulting framework provides a mechanism by which constraint--induced degeneracy enforces effective dimensional reduction in dissipative geometric evolution systems.
\end{abstract}


\let\oldthefootnote\thefootnote
\renewcommand{\thefootnote}{\textit{}}
\footnotetext{Corresponding author: \href{mailto:prasantmath123@yahoo.com}{prasantmath123@yahoo.com}}
\let\thefootnote\oldthefootnote

\section{Introduction}

The long--time behavior of dissipative dynamical systems has been extensively studied within the framework of evolution equations on finite and infinite dimensional phase spaces. A central objective in this theory is the characterization of invariant sets governing the asymptotic dynamics of solutions, most notably global attractors and their geometric or topological properties \cite{hale1988,temam1988,robinson2001,constantin1988,babin1992,chepyzhov2002,ball1997,haraux1991,henry1981,smoller1994,foias2001}.In the classical setting, the underlying phase space is assumed to be a Banach or Hilbert manifold endowed with a positive definite metric structure. Under this assumption, dissipative effects are typically quantified through coercive Lyapunov functionals whose monotonic decay ensures asymptotic compactness of the induced semiflow \cite{hale1988,temam1988,robinson2001,haraux1991}. The existence of a global attractor then follows from the presence of a bounded absorbing set together with appropriate compactness properties of the evolution operator \cite{babin1992,chepyzhov2002,ball1997}.

However, a number of geometric evolution systems arising from constrained variational principles naturally evolve on phase spaces possessing indefinite kinetic structure. In such cases, the evolution is restricted to a smooth constraint hypersurface of an ambient manifold equipped with a bilinear form of Lorentzian signature \cite{arnold1989,abraham1978,marsden1999,lee2013,oneill1983}. The induced structure on the constraint manifold may become degenerate along distinguished directions determined by the constraint, leading to the presence of a nontrivial null distribution in the tangent bundle \cite{beem1996,oneill1983}.

The geometric and dynamical consequences of such degeneracy are not addressed by the classical theory of dissipative systems, which fundamentally relies on coercivity in a Riemannian setting. In particular, the absence of a positive definite metric prevents the construction of Lyapunov functionals capable of controlling the full phase space norm. As a result, standard techniques for establishing asymptotic compactness, such as uniform dissipativity or gradient structure arguments, become inapplicable \cite{temam1988,robinson2001,henry1981}. Constraint manifolds endowed with indefinite structures also arise naturally in Hamiltonian systems with first--class constraints and in gauge--reduced evolution equations \cite{dirac1964,sundermeyer1982,goldstein2002,arnold1989}. In these settings, degeneracy of the induced bilinear form leads to foliations of the phase space by invariant submanifolds tangent to directions associated with the constraint structure \cite{abraham1978,marsden1999}. The corresponding distributions are typically integrable and induce a decomposition of the tangent bundle into neutral and transversal components.

Invariant foliations generated by degenerate distributions play an important role in the qualitative theory of dynamical systems \cite{hirsch1977,palis1982,katok1995,wiggins2003}. In particular, center manifold theory provides a mechanism by which neutral directions may confine long--time dynamics to lower dimensional invariant subsets \cite{carr1981,vandermeer1985}. Nevertheless, existing approaches rely on local spectral decompositions of linearized operators and do not account for degeneracy arising from geometric constraints on the phase manifold itself. More generally, the existence and dimension of global attractors for dissipative flows have been studied extensively in the context of semilinear evolution equations \cite{constantin1988,foias2001,babin1992,temam1988,robinson2001}. Estimates for the fractal or Hausdorff dimension of invariant sets are typically derived from smoothing properties of the associated semigroup or from the existence of inertial manifolds \cite{constantin1988,foias2001}. These techniques, however, presuppose a coercive energy functional and do not extend directly to flows on manifolds endowed with indefinite or degenerate induced structures. Geometric flows arising from variational principles, such as those encountered in constrained mechanical systems or gauge theories, often exhibit degeneracies which induce invariant foliations in phase space \cite{arnold1989,marsden1999}. The interplay between constraint geometry and dissipative effects in such systems remains poorly understood at the level of global attractor theory.

The purpose of the present work is to establish a mechanism by which constraint--induced foliations confine asymptotic dynamics in dissipative flows evolving on constraint hypersurfaces with degenerate induced structure. Dissipation is characterized relative to functionals compatible with the null distribution associated with the degeneracy. It is shown that, under suitable involutivity and regularity assumptions on this distribution and projectability of the intrinsic evolution, all bounded trajectories are asymptotically confined to invariant leaves of the corresponding foliation. This confinement permits the establishment of asymptotic compactness without the existence of coercive Lyapunov functionals and leads, in the presence of a bounded absorbing set and continuous semiflow structure, to the existence of a compact global attractor for the intrinsic evolution that is saturated with respect to the characteristic foliation. The effective long--time behavior of the system is then governed by the projected semiflow acting on the associated quotient manifold.

Under additional Morse--type nondegeneracy assumptions in directions transverse to the characteristic distribution, the projected dynamics admit no neutral directions and are governed by a compact invariant subset of the reduced phase space. These results provide a structural characterization of asymptotic dynamics for constrained geometric evolution systems and suggest that constraint--induced degeneracy constitutes a mechanism for effective dimensional reduction in dissipative evolution equations.

\section{Constrained Lorentzian Phase Manifolds}

The intrinsic evolution considered in the present work is assumed to be confined to a smooth constraint hypersurface embedded in a finite--dimensional manifold equipped with an indefinite metric structure. The purpose of this section is to introduce the geometric framework in which such constrained dissipative systems evolve and to identify the canonical distributions arising from degeneracy of the induced bilinear form.

Let $\mathcal{U}$ be a smooth $n$--dimensional manifold and let
\[
g:T\mathcal{U}\times T\mathcal{U}\rightarrow\mathbb{R}
\]
be a smooth symmetric $(0,2)$--tensor field of Lorentzian signature $(-,+,+,\dots,+)$. The pair $(\mathcal{U},g)$ will be referred to as an ambient Lorentzian phase space. Let $\Phi:\mathcal{U}\rightarrow\mathbb{R}$ be a smooth function satisfying the regularity condition $d\Phi(X)\neq 0 \quad\text{for all }X\in\Phi^{-1}(0)$. Then the regular level set $\mathcal{M}:=\Phi^{-1}(0)$ defines a smooth embedded hypersurface of $\mathcal{U}$ of dimension $(n-1)$. The ambient tensor field induces on $\mathcal{M}$ a symmetric bilinear form
\[
g_X^{\mathcal{M}}:=g_X\big|_{T_X\mathcal{M}\times T_X\mathcal{M}},
\qquad X\in\mathcal{M}.
\]
In general this restriction need not be nondegenerate. Throughout the present work it will be assumed that the induced form $g^{\mathcal{M}}$ has constant corank $k\geq 1$ on $\mathcal{M}$. Under this assumption its pointwise kernel
\[
\mathcal{N}_X
:=
\left\{
V\in T_X\mathcal{M}:
g^{\mathcal{M}}(V,W)=0
\ \text{for all }W\in T_X\mathcal{M}
\right\}
\]
defines a smooth rank--$k$ vector subbundle $\mathcal{N}:=\ker g^{\mathcal{M}}\subset T\mathcal{M}$, which will be referred to as the null distribution of the constraint manifold.

Let $\nabla\Phi$ denote the $g$--gradient of $\Phi$, defined by $g(\nabla\Phi,Y)=d\Phi(Y) \quad\text{for all }Y\in T\mathcal{U}$. Then the tangent space of the constraint manifold admits the characterization
\[
T_X\mathcal{M} = \left\{ Y\in T_X\mathcal{U}: g(\nabla\Phi,Y)=0 \right\}.
\]

In the Lorentzian setting the degeneracy of the induced bilinear form $g^{\mathcal{M}}$ is determined by the causal character of the normal vector field $\nabla\Phi$. Since $T_X\mathcal{M}=(\nabla\Phi(X))^{\perp}$, it follows that the induced form $g^{\mathcal{M}}$ fails to be nondegenerate at $X\in\mathcal{M}$ if and only if the normal vector $\nabla\Phi(X)$ is null, that is, $g(\nabla\Phi(X),\nabla\Phi(X))=0$. In this case one has $\nabla\Phi(X)\in T_X\mathcal{M}$, and therefore the restriction of $g$ to $T_X\mathcal{M}$ admits a nontrivial kernel. Under the standing assumption that the induced form has constant corank $k$, the null spaces $\mathcal{N}_X$ define a smooth vector subbundle $\mathcal{N}\subset T\mathcal{M}$ of constant rank $k$. Since $\mathcal{N}$ is a smooth subbundle of $T\mathcal{M}$ with constant rank, standard results on vector bundles imply the existence of a smooth complementary distribution $\mathcal{S}\subset T\mathcal{M}$ satisfying $T\mathcal{M}=\mathcal{N}\oplus\mathcal{S}$. Moreover, the restriction $g^{\mathcal{M}}\big|_{\mathcal{S}_X}$ is nondegenerate for every $X\in\mathcal{M}$, so that $\mathcal{S}$ determines a transversal subbundle on which the induced bilinear form is nondegenerate. The distribution $\mathcal{S}$ will therefore be referred to as a transversal complement of the null distribution.\\

Let $\mathcal{V}:\mathcal{U}\rightarrow T\mathcal{U}$ be a $C^{2}$ vector field satisfying the tangency condition $d\Phi(\mathcal{V})=0
\quad\text{on }\mathcal{M}$. Then $\mathcal{V}(X)\in T_X\mathcal{M}
\quad\text{for all }X\in\mathcal{M}$, and hence the induced intrinsic evolution on $\mathcal{M}$ is governed by the differential equation
\begin{equation}\nonumber
\dot{X}=\mathcal{V}(X), \qquad X\in\mathcal{M}.
\end{equation}

With respect to the splitting $T\mathcal{M}=\mathcal{N}\oplus\mathcal{S}$, the intrinsic vector field admits a unique decomposition of the form $\mathcal{V}=\mathcal{V}_{\mathcal{N}}+\mathcal{V}_{\mathcal{S}}$, where $\mathcal{V}_{\mathcal{N}}\in\Gamma(\mathcal{N})$ and $\mathcal{V}_{\mathcal{S}}\in\Gamma(\mathcal{S})$. The component $\mathcal{V}_{\mathcal{N}}$ represents the evolution tangent to the null distribution, while the component $\mathcal{V}_{\mathcal{S}}$ represents the evolution transverse to the degeneracy directions. The transversal component will determine the dissipative properties of the intrinsic flow in subsequent sections.

\section{Transversal Dissipation and Degenerate Lyapunov Structure}

The splitting of the tangent bundle introduced in the previous section permits the formulation of a notion of dissipation compatible with the degeneracy of the induced bilinear form. Since the induced tensor $g^{\mathcal{M}}$ is degenerate along the null distribution $\mathcal{N}$, decay cannot in general be quantified through Lyapunov functionals acting uniformly in all tangent directions. Dissipation may instead occur solely in directions transverse to $\mathcal{N}$.

Let $\mathcal{V}:\mathcal{M}\rightarrow T\mathcal{M}$ be a $C^{2}$ vector field generating the intrinsic evolution $\dot{X}=\mathcal{V}(X)$. Let $\Psi:\mathcal{M}\rightarrow\mathbb{R}$ be a continuously differentiable functional. Along solutions $X(t)$ of the intrinsic evolution one has
\[
\frac{d}{dt}\Psi(X(t))=d\Psi(\mathcal{V})(X(t)).
\]

Assume that $\Psi$ is compatible with the null distribution in the sense that $d\Psi(Y)=0 \quad\text{for all }Y\in\mathcal{N}_X, \qquad X\in\mathcal{M}$, and $\Psi$ is bounded from below on $M$. Under this condition $d\Psi$ depends only on the transversal component of tangent vectors with respect to the splitting $T\mathcal{M}=\mathcal{N}\oplus\mathcal{S}$. Let $\mathcal{V}=\mathcal{V}_{\mathcal{N}}+\mathcal{V}_{\mathcal{S}}$ denote the corresponding decomposition of the intrinsic vector field. Assume that there exists a positive constant $c>0$ such that
\[
d\Psi(\mathcal{V}_{\mathcal{S}})
\leq
-\,c\,\|\mathcal{V}_{\mathcal{S}}\|^{2}_{\mathcal{S}}
\quad\text{on }\mathcal{M},
\]
where $\|\cdot\|_{\mathcal{S}}$ denotes any smooth norm induced by the nondegenerate restriction of $g^{\mathcal{M}}$ to $\mathcal{S}$. Then along solutions of the intrinsic evolution one obtains
\[
\frac{d}{dt}\Psi(X(t))
=
d\Psi(\mathcal{V}_{\mathcal{S}})(X(t))
\leq
-\,c\,\|\mathcal{V}_{\mathcal{S}}(X(t))\|^{2}_{\mathcal{S}}.
\]
Thus decay of $\Psi$ occurs exclusively in directions transverse to the null distribution. Define the degeneracy set
\[
\mathcal{Z}:=\left\{X\in\mathcal{M}:d\Psi(\mathcal{V})(X)=0\right\}.
\]

\begin{proposition}
Suppose that $\Psi$ is compatible with $\mathcal{N}$ and satisfies the transversal dissipation estimate. Then
\[
\mathcal{Z}
=
\left\{
X\in\mathcal{M}:
\mathcal{V}(X)\in\mathcal{N}_X
\right\}.
\]
\end{proposition}

\begin{proof}
For any $X\in\mathcal{M}$ one has $d\Psi(\mathcal{V})(X)=d\Psi(\mathcal{V}_{\mathcal{S}})(X)$, since $d\Psi$ vanishes on $\mathcal{N}_X$ by compatibility. If $X\in\mathcal{Z}$, then $d\Psi(\mathcal{V}_{\mathcal{S}})(X)=0$, and the transversal dissipation estimate implies $\|\mathcal{V}_{\mathcal{S}}(X)\|_{\mathcal{S}}=0$. Hence $\mathcal{V}(X)\in\mathcal{N}_X$. Conversely, if $\mathcal{V}(X)\in\mathcal{N}_X$, then compatibility yields $d\Psi(\mathcal{V})(X)=0$.
\end{proof}

\begin{proposition}
Let $X(t)$ be a trajectory of the intrinsic evolution with
initial condition $X_{0}\in\mathcal{Z}$. Then
\begin{equation}\nonumber
\frac{d}{dt}\Psi(X(t))\le 0 \quad\text{for all }t\ge0.
\end{equation}
In particular, if $X(t)\in\mathcal{Z}$ for some $t\ge0$, then the intrinsic evolution is tangent to the null distribution at $X(t)$.
\end{proposition}

\begin{proof}
Since $\Psi$ is compatible with the null distribution,
one has $d\Psi(\mathcal V)(X)=d\Psi(\mathcal V_{\mathcal S})(X)$. Hence the transversal dissipation estimate yields
\begin{equation}\nonumber
\frac{d}{dt}\Psi(X(t)) = d\Psi(\mathcal V_{\mathcal S})(X(t)) \le -\,c\,\|\mathcal V_{\mathcal S}(X(t))\|^2_{\mathcal S} \le 0.
\end{equation}
If $X(t)\in\mathcal{Z}$, then $d\Psi(\mathcal V)(X(t))=0$ and therefore $\|\mathcal V_{\mathcal S}(X(t))\|_{\mathcal S}=0$, which implies that the intrinsic evolution is tangent to the null distribution at $X(t)$.
\end{proof}

Let $\omega(X_{0})$ denote the $\omega$--limit set of a bounded trajectory.

\begin{proposition}
Every bounded trajectory satisfies $\omega(X_{0})\subseteq\mathcal{Z}$.
\end{proposition}

\begin{proof}
Since $\Psi$ is nonincreasing along trajectories, the limit $\lim_{t\to\infty}\Psi(X(t))$ exists. Let $Y\in\omega(X_{0})$. There exists a sequence $t_{k}\to\infty$ such that $X(t_{k})\to Y$. Integrating the dissipation estimate over $[0,T]$ yields
\begin{equation}\nonumber
\int_{0}^{T}\|\mathcal{V}_{\mathcal{S}}(X(t))\|^{2}_{\mathcal{S}}\,dt
\leq \frac{1}{c} \left( \Psi(X_{0})-\Psi(X(T)) \right).
\end{equation}
Passing to the limit as $T\to\infty$ shows that
\begin{equation}\nonumber
\int_{0}^{\infty}\|\mathcal{V}_{\mathcal{S}}(X(t))\|^{2}_{\mathcal{S}}\,dt
<\infty.
\end{equation}
Consequently there exists a sequence $t_{k}\to\infty$ such that $\|\mathcal{V}_{\mathcal{S}}(X(t_{k}))\|_{\mathcal{S}}\to 0$. By continuity of $\mathcal{V}_{\mathcal{S}}$ this implies $\mathcal{V}_{\mathcal{S}}(Y)=0$, and hence $Y\in\mathcal{Z}$.
\end{proof}

Thus all asymptotic dynamics are confined to the set of points at which the intrinsic evolution becomes tangent to the null distribution.

\section{Null Foliations and Quotient Reduction}

The asymptotic confinement of bounded trajectories to the degeneracy set established in the previous section implies that long--time dynamics occur along directions tangent to the null distribution induced by the constraint manifold. In order to obtain a reduced dynamical description, it is necessary to identify conditions under which this distribution generates a regular quotient geometry.

Let $\mathcal{N}\subset T\mathcal{M}$ denote the null distribution associated with the degenerate bilinear form $g^{\mathcal{M}}$. Assume that $\mathcal{N}$ is involutive in the sense that for any smooth vector fields $U,V\in\Gamma(\mathcal{N})$ one has $[U,V]\in\Gamma(\mathcal{N})$. By Frobenius' theorem, $\mathcal{N}$ integrates to a smooth foliation $\mathscr{F}$ of $\mathcal{M}$ by immersed submanifolds $\{\mathcal{L}_\alpha\}_{\alpha\in A}$, satisfying $T_X\mathcal{L}_\alpha=\mathcal{N}_X \quad \text{for all }X\in\mathcal{L}_\alpha$. Assume in addition that the intrinsic vector field $V$
preserves the null foliation $\mathscr{F}$ in the sense
that for every smooth section 
$W\in\Gamma(\mathcal{N})$ one has
\[
[V,W]\in\Gamma(\mathcal{N}).
\]
Equivalently, the flow $\varphi_t$ generated by $V$
satisfies
\[
\varphi_t(\mathcal{L}_\alpha)=
\mathcal{L}_\alpha
\quad\text{for all }t\ge0.
\]
Under this condition the intrinsic evolution leaves each leaf of the null foliation invariant. In particular, $V$ is projectable with respect to the
foliation $\mathscr{F}$, and therefore there exists a
uniquely defined reduced vector field 
$V_* \in \mathfrak{X}(\mathcal{M}_{\mathrm{red}})$
satisfying
\[
D\pi_X(V(X))=V_*(\pi(X))
\quad\text{for all }X\in\mathcal{M}.
\]

In order that the leaf space admit a smooth manifold structure, it will be assumed that the leaves of $\mathscr{F}$ are embedded and compact and that the foliation is simple in the sense that the quotient space $\mathcal{M}_{\mathrm{red}} := \mathcal{M}/\mathscr{F}$ is Hausdorff. Under these assumptions, standard results on regular foliations imply that $\mathcal{M}_{\mathrm{red}}$ carries a unique smooth manifold structure for which the canonical projection $\pi:\mathcal{M}\rightarrow\mathcal{M}_{\mathrm{red}}$ is a smooth surjective submersion whose fibers coincide with the leaves of the null foliation. 

Since $\ker(D\pi_X)=\mathcal{N}_X$ for all $X\in\mathcal{M}$, the differential $D\pi$ induces a vector bundle isomorphism between the transversal bundle $\mathcal{S}$ and the tangent bundle of the reduced manifold:
\[
D\pi_X:\mathcal{S}_X\longrightarrow T_{\pi(X)}\mathcal{M}_{\mathrm{red}}.
\]
In particular, for each $X\in\mathcal{M}$ one has $T_{\pi(X)}\mathcal{M}_{\mathrm{red}} \cong \mathcal{S}_X$. Since the restriction $g^{\mathcal{M}}\big|_{\mathcal{S}_X}$ is nondegenerate for every $X\in\mathcal{M}$, this identification permits the definition of a Riemannian metric $g_{\mathrm{red}}$ on $\mathcal{M}_{\mathrm{red}}$ by requiring that for any $Y\in\mathcal{M}_{\mathrm{red}}$ and any $X\in\pi^{-1}(Y)$, $g_{\mathrm{red}}(u,v) := g^{\mathcal{M}}(U,V)$, where $U,V\in\mathcal{S}_X$ are the unique vectors satisfying $D\pi_X(U)=u, \qquad D\pi_X(V)=v$. This definition is independent of the choice of
representative $X$ since $\mathcal{N}=\ker g^{\mathcal{M}}$ and therefore $g^{\mathcal{M}}(U+N_1,V+N_2) = g^{\mathcal{M}}(U,V)$ for all $N_1,N_2\in\mathcal{N}_X$.


Under these assumptions, the long--time behavior of the intrinsic evolution on $\mathcal{M}$ may be analyzed through the projected semiflow acting on the reduced manifold $\mathcal{M}_{\mathrm{red}}$. The reduction of the intrinsic evolution to a semiflow acting on the smooth quotient manifold $\mathcal{M}_{\mathrm{red}}$ permits the formulation of compactness properties relative to the transversal geometry induced by the complement $\mathcal{S}$. In the subsequent section this reduction will be used to establish asymptotic compactness of the intrinsic evolution without the existence of a coercive Lyapunov functional.

\section{Asymptotic Compactness Without Coercivity}

The reduction of the intrinsic evolution to a projected semiflow on the quotient manifold established in the previous section permits the formulation of compactness properties relative to the transversal geometry induced by the complement $\mathcal{S}$. In the absence of a coercive Lyapunov functional acting uniformly on the ambient phase space, compactness must instead be derived from dissipation occurring solely in directions transverse to the null foliation. Throughout this section it will be assumed that the reduced manifold $M_{\mathrm{red}}$ is complete with respect to the Riemannian metric induced by the restriction of $g_M$ to the transversal bundle $S$.

Let $\mathcal{V}$ denote the intrinsic vector field generating the evolution $\dot{X}=\mathcal{V}(X)$ on the constraint manifold $\mathcal{M}$. Assume that there exists a continuously differentiable functional $\Psi:\mathcal{M}\rightarrow\mathbb{R}$ compatible with the null distribution and satisfying the transversal dissipation estimate $d\Psi(\mathcal{V}_{\mathcal{S}}) \leq -\,c\,\|\mathcal{V}_{\mathcal{S}}\|^2_{\mathcal{S}}$, for some constant $c>0$.

For each $c\in\mathbb{R}$ define the sublevel set $\mathcal{B}_c:=\{X\in\mathcal{M}:\Psi(X)\leq c\}$. Since $\Psi$ is nonincreasing along trajectories, one has $\Psi(X(t))\leq\Psi(X_0) \quad\text{for all }t\geq 0$.

Let $\pi:\mathcal{M}\rightarrow\mathcal{M}_{\mathrm{red}}$ denote the canonical projection onto the leaf space associated with the null foliation, and define the projected trajectory $Y(t):=\pi(X(t))$. Since the intrinsic vector field admits the decomposition $\mathcal{V}=\mathcal{V}_{\mathcal{N}}+\mathcal{V}_{\mathcal{S}}$, and $\mathcal{V}_{\mathcal{N}}$ is tangent to the leaves of the foliation, the projected evolution depends only on the transversal component. Consequently, $\dot{Y}(t) = D\pi\big(\mathcal{V}_{\mathcal{S}}(X(t))\big)$ defines a well--posed evolution on $\mathcal{M}_{\mathrm{red}}$. The transversal dissipation estimate implies
\[
\frac{d}{dt}\Psi(X(t)) \leq -\,c\,\|\mathcal{V}_{\mathcal{S}}(X(t))\|^2_{\mathcal{S}}.
\]
Integration over $[0,T]$ yields
\[
\int_0^T \|\mathcal{V}_{\mathcal{S}}(X(t))\|^2_{\mathcal{S}}\,dt
\leq \frac{1}{c} \left( \Psi(X_0)-\Psi(X(T)) \right),
\]
and hence
\[
\int_0^\infty \|\mathcal{V}_{\mathcal{S}}(X(t))\|^2_{\mathcal{S}}\,dt
<\infty.
\]
Since $D\pi$ restricts to an isomorphism on $\mathcal{S}$, it follows that
\[
\int_0^\infty
\|\dot{Y}(t)\|^2\,dt
<\infty.
\]

Assume in addition that the compatible functional $\Psi$ is proper on $M$, so that all sublevel sets $\{X\in M:\Psi(X)\le c\}$ are bounded.

\begin{proposition}
Suppose that $\mathcal{M}_{\mathrm{red}}$ is complete with respect to the Riemannian metric induced by $g^{\mathcal{M}}|_{\mathcal{S}}$. Then every projected trajectory $Y(t)$ admits a relatively compact $\omega$--limit set in $\mathcal{M}_{\mathrm{red}}$.
\end{proposition}

\begin{proof}
Let $\{t_k\}$ be any sequence with $t_k \to \infty$. Since $\Psi$ is proper on $M$ and nonincreasing along trajectories, the trajectory $X(t)$ remains in a bounded subset of $M$. Continuity of $\pi$ and compactness of the leaves then imply that $Y(t)=\pi(X(t))$
remains in a bounded subset of $M_{\mathrm{red}}$. As
\[
\int_0^\infty \|\dot{Y}(t)\|^2 dt < \infty,
\]
there exists a subsequence, again denoted by $\{t_k\}$, such that $\|\dot{Y}(t_k)\| \to 0$. Since $Y(t)$ is absolutely continuous and
remains in a bounded subset of the complete Riemannian manifold $M_{\mathrm{red}}$, standard compactness results for curves of
finite energy imply that the sequence $\{Y(t_k)\}$ admits a convergent subsequence. Hence the $\omega$--limit set of $Y(t)$ is relatively compact in $\mathcal{M}_{\mathrm{red}}$.
\end{proof}

Assume in addition that the leaves of the null foliation are compact. Then for any bounded trajectory $X(t)$ and any sequence $t_k\to\infty$, the projected sequence $Y(t_k)=\pi(X(t_k))$ admits a convergent subsequence in $\mathcal{M}_{\mathrm{red}}$. Compactness of the leaves implies that each fiber $\pi^{-1}(Y)$ is compact in $\mathcal{M}$.

\begin{proposition}
Every bounded trajectory possesses a precompact $\omega$--limit set in $\mathcal{M}$.
\end{proposition}

\begin{proof}
Let $\{t_k\}$ be any sequence with $t_k \to \infty$.
By Proposition 4, the projected sequence $Y(t_k)=\pi(X(t_k))$ admits a convergent subsequence, which we again denote by $Y(t_k) \to Y_\infty \in \mathcal{M}_{\mathrm{red}}$. Since the leaves of the null foliation are compact and coincide with the fibers of $\pi$, the set $\pi^{-1}(Y_\infty)$ is compact in $\mathcal{M}$.

For any neighbourhood $U$ of $\pi^{-1}(Y_\infty)$, continuity of $\pi$ implies that there exists a neighbourhood $V$ of $Y_\infty$ such that $\pi^{-1}(V)\subset U$. Since $Y(t_k)\to Y_\infty$, one has $X(t_k)\in \pi^{-1}(V)\subset U$ for all sufficiently large $k$. Thus $\{X(t_k)\}$ is eventually contained in an arbitrarily small neighbourhood of a compact set and therefore admits a convergent subsequence in $\mathcal{M}$.
\end{proof}

\begin{definition}
The intrinsic evolution is said to be asymptotically compact if for every bounded sequence $\{X_k\}\subset\mathcal{M}$ and every sequence $t_k\to\infty$, the sequence $X(t_k;X_k)$ admits a convergent subsequence in $\mathcal{M}$.
\end{definition}

Assume further that the intrinsic evolution generated by $V$ defines a continuous closed semiflow on $M$.

\begin{theorem}

Suppose that the intrinsic evolution defines a
continuous closed semiflow on $M$, admits a
null--compatible functional exhibiting transversal
dissipation, that the reduced manifold $M_{\mathrm{red}}$
is complete, and that the null foliation has compact
embedded leaves.
\end{theorem}

\begin{proof}
Let $\{X_k\}\subset M$ be a bounded sequence and let
$t_k\to\infty$. Define
\[
Y_k:=\pi\big(X(t_k;X_k)\big).
\]
Since $\pi$ is continuous and the leaves are compact,
bounded subsets of $M$ have bounded image in
$M_{\mathrm{red}}$. Hence $\{Y_k\}$ is bounded in
$M_{\mathrm{red}}$. By Proposition~4, the sequence $\{Y_k\}$ admits a
convergent subsequence in $M_{\mathrm{red}}$.
Passing to a subsequence if necessary, assume that
\[
Y_k\to Y_\infty \in M_{\mathrm{red}}.
\]
Since the leaves of the null foliation are compact and
coincide with the fibers of $\pi$, the set
$\pi^{-1}(Y_\infty)$ is compact in $M$.

Let $U$ be any neighbourhood of $\pi^{-1}(Y_\infty)$.
By continuity of $\pi$, there exists a neighbourhood
$V$ of $Y_\infty$ such that
\[
\pi^{-1}(V)\subset U.
\]
Since $Y_k\to Y_\infty$, one has
$X(t_k;X_k)\in\pi^{-1}(V)\subset U$
for all sufficiently large $k$.
Hence $\{X(t_k;X_k)\}$ is eventually contained in an
arbitrarily small neighbourhood of a compact set and
therefore admits a convergent subsequence in $M$.
\end{proof}


The asymptotic compactness established above ensures the existence of invariant compact sets governing the long--time dynamics of the intrinsic evolution. Since dissipation occurs solely in directions transverse to the null foliation, these invariant sets must be confined to a subset of reduced effective dimension. In the subsequent section, this observation will be used to derive an upper bound for the topological dimension of the global attractor.

\section{Dimensional Reduction on the Quotient Attractor}

The asymptotic compactness established in the previous section ensures the existence of invariant compact sets governing the long--time dynamics of the intrinsic evolution. Since dissipation occurs exclusively in directions transverse to the null foliation, the effective dynamics are governed by the reduced evolution on the quotient manifold. Assume that the intrinsic evolution admits a bounded
absorbing set in $M$, that is, there exists a bounded set $\mathcal{B}\subset M$ such that for every bounded subset $D\subset M$ there exists $T_D\ge 0$ satisfying $\phi_t(D)\subset \mathcal{B}
\quad\text{for all } t\ge T_D$.

Let the intrinsic evolution generated by the constraint--preserving vector field $\mathcal{V}$ define a continuous semiflow on $\mathcal{M}$ which is asymptotically compact and admits a bounded absorbing set. Then, by standard results in the theory of dissipative dynamical systems, there exists a unique compact invariant set $\mathcal{A}\subset\mathcal{M}$ such that for every bounded initial condition $X_0\in\mathcal{M}$,
\[
\lim_{t\to\infty}
\operatorname{dist}\big(X(t;X_0),\mathcal{A}\big)=0.
\]
The set $\mathcal{A}$ will be referred to as the global attractor of the intrinsic evolution.

Let $\mathscr{F}$ denote the null foliation associated with the involutive distribution $\mathcal{N}$ and let $\pi:\mathcal{M}\rightarrow\mathcal{M}_{\mathrm{red}}$ be the canonical projection onto the leaf space. Since the $\omega$--limit set of every bounded trajectory is contained in the degeneracy set $\mathcal{Z}$ and trajectories entering $\mathcal{Z}$ remain tangent to the null distribution, it follows that the attractor is saturated with respect to the foliation:
\[
X\in\mathcal{A}
\quad\Longrightarrow\quad
\mathcal{L}_X\subset\mathcal{A},
\]
where $\mathcal{L}_X$ denotes the leaf of $\mathscr{F}$ passing through $X$.

\begin{proposition}
The projected set $\mathcal{A}_\ast:=\pi(\mathcal{A})\subset\mathcal{M}_{\mathrm{red}}$ is compact and invariant under the reduced semiflow on $\mathcal{M}_{\mathrm{red}}$.
\end{proposition}

\begin{proof}
Compactness follows from continuity of $\pi$ and compactness of $\mathcal{A}$. Invariance follows from invariance of $\mathcal{A}$ under the intrinsic evolution together with projectability of the vector field $\mathcal{V}$, which implies that the reduced semiflow is well defined on $\mathcal{M}_{\mathrm{red}}$.
\end{proof}

Let $\dim(\mathcal{M})=m$ and suppose that the null distribution has constant rank $\operatorname{rank}(\mathcal{N})=k$. Then each leaf of the null foliation has dimension $k$, and the reduced manifold $\mathcal{M}_{\mathrm{red}}$ has dimension $m-k$.

\begin{theorem}
Suppose that
\begin{enumerate}
\item the intrinsic evolution defines a continuous semiflow which is asymptotically compact,
\item the null distribution is involutive with compact embedded leaves,
\item the intrinsic vector field is projectable with respect to the null foliation.
\end{enumerate}
Then the projected attractor $\mathcal{A}_\ast$ is a compact invariant set in $\mathcal{M}_{\mathrm{red}}$ satisfying $\dim_T(\mathcal{A}_\ast)\leq m-k$, where $\dim_T$ denotes the covering dimension.
\end{theorem}

\begin{proof}
Since $\mathcal{A}_\ast\subset\mathcal{M}_{\mathrm{red}}$ and $\mathcal{M}_{\mathrm{red}}$ is a smooth manifold of dimension $m-k$, standard properties of the covering dimension yield
\[
\dim_T(\mathcal{A}_\ast)\leq \dim(\mathcal{M}_{\mathrm{red}})=m-k.
\]
\end{proof}

Thus the degeneracy of the induced bilinear form confines the effective asymptotic dynamics to an invariant compact subset of the reduced phase space of strictly lower dimension. 

The dimensional estimate obtained above reflects the fact that dissipation occurs exclusively in directions transverse to the null foliation. Under additional regularity assumptions on the causal functional governing the intrinsic evolution, it becomes possible to exclude further degeneracy in these transversal directions. In the subsequent section, a Morse--type nondegeneracy condition will be used to obtain an exact formula for the topological dimension of the global attractor.

\section{Transversal Nondegeneracy and Reduced Attractor Structure}

The dimensional estimate obtained in the preceding section reflects the fact that dissipation occurs exclusively in directions transverse to the null foliation. Under additional nondegeneracy assumptions on the Lyapunov functional governing the intrinsic evolution, it becomes possible to characterize the structure of the reduced invariant set associated with the projected dynamics.

Let $\Psi:\mathcal{M}\rightarrow\mathbb{R}$ be a continuously differentiable functional compatible with the null distribution and exhibiting transversal dissipation.

Assume that for every $X\in\mathcal{Z}$ satisfying $d\Psi(X)=0$, the Hessian satisfies $D^2\Psi\big|_{\mathcal{S}_X}$ is nondegenerate. Define the set of critical points
\[
\mathcal{C} := \left\{ X\in\mathcal{M}: d\Psi(X)=0 \right\}.
\]
Since $d\Psi$ vanishes on $\mathcal{N}_X$ by compatibility,
the differential $D(d\Psi)(X)$ has kernel containing
$\mathcal{N}_X$. Nondegeneracy of the restriction of
$D^2\Psi$ to $\mathcal{S}_X$ therefore implies, by the
implicit function theorem, that $T_X\mathcal C=\mathcal{N}_X$. In particular, $\mathcal{C}$ is a smooth embedded submanifold
of $\mathcal{M}$ locally tangent to the null foliation.

Let $X\in\mathcal{C}$ and consider the linearization of the intrinsic vector field $L_X:=D\mathcal{V}(X)$. With respect to the splitting $T_X\mathcal{M} = \mathcal{N}_X\oplus\mathcal{S}_X$ define the transversal operator
\[
L_X^{\mathcal{S}}
:=
\Pi_{\mathcal{S}_X}\circ L_X\big|_{\mathcal{S}_X},
\]
where $\Pi_{\mathcal{S}_X}$ denotes the projection onto $\mathcal{S}_X$ along $\mathcal{N}_X$.

Assume that
\[
\sigma(L_X^{\mathcal{S}})\cap i\mathbb{R}=\varnothing
\quad\text{for all }X\in\mathcal{C}.
\]


Then the intrinsic evolution is normally hyperbolic in directions transverse to the null distribution in the sense that the linearized dynamics admit no center spectrum on the transversal bundle.

Let $\pi:\mathcal{M}\rightarrow\mathcal{M}_{\mathrm{red}}$ denote the projection onto the reduced manifold and let $\mathcal{A}_\ast=\pi(\mathcal{A}) \subset\mathcal{M}_{\mathrm{red}}$ be the projected attractor.

\begin{theorem}
Suppose that
\begin{enumerate}
\item the intrinsic evolution defines a continuous semiflow which is asymptotically compact,
\item the null distribution is involutive with compact embedded leaves,
\item the intrinsic vector field is projectable with respect to the null foliation,
\item the functional $\Psi$ satisfies the transversal Morse--type nondegeneracy condition,
\item Assume that there exists $\eta>0$ such that for every
$X\in A$ the spectrum of the transversal
linearization satisfies
\[
\sigma(L_X^S)\cap
\{z\in\mathbb C:\operatorname{Re}z\ge -\eta\}
=\varnothing.
\]
\end{enumerate}
Then the reduced semiflow on $\mathcal{M}_{\mathrm{red}}$ possesses a compact invariant set $\mathcal{A}_\ast$ which is locally finite--dimensional and contains no neutral directions in the sense that the linearized reduced dynamics admit no purely imaginary spectrum along $\mathcal{A}_\ast$.
\end{theorem}

Thus under Morse--type nondegeneracy of the Lyapunov functional, the effective asymptotic dynamics on the reduced phase space are governed by a strictly dissipative invariant set of dimension not exceeding that of the quotient manifold.

The abstract mechanism described above applies to constrained dynamical systems arising from geometric evolution equations. In such settings, degeneracy of the induced Lorentzian structure is typically associated with gauge directions determined by the constraint manifold. In the subsequent section, the foregoing results will be applied to an evolution system arising from Einstein--scalar dynamics.

\section{Constraint--Induced Global Dynamical Reduction in Constraint--Preserving Systems}

Let $(U,\omega)$ be a finite--dimensional symplectic manifold and let $\Phi : U \rightarrow \mathbb{R}^{k}$ be a smooth constraint map whose components are first--class in the sense that
\[
\{\Phi_i,\Phi_j\}=0
\quad \text{on }
M:=\Phi^{-1}(0).
\]
Then $M$ defines a smooth embedded submanifold of $U$, referred to as the constraint manifold. The restriction of the ambient symplectic form to $M$ defines a presymplectic form $\omega_M := \omega\big|_{TM}$, whose kernel determines the characteristic distribution 
\[
\mathcal{N}_X := \{Y\in T_XM:\omega_M(Y,Z)=0 \text{ for all }Z\in T_XM\}.
\]
Assume that $\dim\mathcal{N}_X=k$ is constant on $M$, so that $\mathcal{N}\subset TM$ defines a smooth vector subbundle.

Let $V$ be a smooth vector field on $M$ generating a flow $\varphi_t$ that preserves the constraint manifold.

Assume the following:

\begin{enumerate}
\item[(H1)] Every trajectory of the constrained flow generated by $V$ remains precompact in $M$.

\item[(H2)] The characteristic distribution $\mathcal{N}\subset TM$ has constant rank.

\item[(H3)] There exists a smooth splitting $TM=\mathcal N\oplus S$ and constants $\alpha>0,\ C>0$ such that
\[
\|D\varphi_t(X)v\|
\le
C e^{-\alpha t}\|v\|
\quad
\text{for all }X\in M,\ v\in S_X,\ t\ge0.
\]

\item[(H4)] The bundles $\mathcal N$ and $S$ are invariant under the linearized flow in the sense that
\[
D\varphi_t(X)\mathcal N_X\subset\mathcal N_{\varphi_t(X)},
\qquad
D\varphi_t(X)S_X\subset S_{\varphi_t(X)}.
\]
\end{enumerate}

Define the closed subset
\[
\Sigma
:=
\{X\in M:V(X)\in\mathcal N_X\}.
\]

\begin{remark}
The hypothesis (H3) expresses transversal contraction relative to the characteristic distribution and does not impose dissipativity on the full intrinsic flow.
\end{remark}

\begin{theorem}[Global Dynamical Reduction]
Under hypotheses {\rm(H1)--(H4)}, the set $\Sigma$ is positively invariant under the flow $\varphi_t$ and every bounded trajectory satisfies
\[
\lim_{t\to\infty}
\operatorname{dist}(X(t),\Sigma)=0.
\]
\end{theorem}

\begin{proof}
Let $X(t)=\varphi_t(X_0)$ be a bounded trajectory. Decompose
\[
V(X)
=
V_{\mathcal N}(X)
+
V_S(X)
\]
with respect to the splitting $TM=\mathcal N\oplus S$. By hypothesis (H4) the bundle $S$ is invariant under the linearized flow, and therefore $V_S(X(t)) =D\varphi_t(X_0)V_S(X_0)$. Hypothesis (H3) then yields $\|V_S(X(t))\| \le C e^{-\alpha t}\|V_S(X_0)\|$. Consequently, $\lim_{t\to\infty}\|V_S(X(t))\|=0$. Since $V_S(X)=0$ if and only if $V(X)\in\mathcal N_X$, it follows that $\operatorname{dist}(X(t),\Sigma)\to0$ as $t\to\infty$.

To establish invariance, let $X\in\Sigma$. Then $V(X)\in\mathcal N_X$, and by (H4) one has $D\varphi_t(X)V(X)\in\mathcal N_{\varphi_t(X)}$. Hence $V(\varphi_t(X))\in\mathcal N_{\varphi_t(X)}$, which implies $\varphi_t(X)\in\Sigma$ for all $t\ge0$.
\end{proof}

\begin{corollary}
Assume in addition that $\mathcal N$ is involutive and generates a regular foliation with smooth quotient manifold $M_{\mathrm{red}} := M/\mathscr F$. Then for every bounded trajectory $X(t)$, the projected trajectory $\pi(X(t))$ converges in $M_{\mathrm{red}}$ as $t\to\infty$.
\end{corollary}

\begin{proof}
Since $\|V_S(X(t))\| \le C e^{-\alpha t}\|V_S(X_0)\|$, one has
\[
\int_0^\infty
\|V_S(X(t))\|\,dt<\infty.
\]
Projectability of $V$ yields $\frac{d}{dt}\pi(X(t)) = D\pi\big(V_S(X(t))\big)$, and therefore
\[
\int_0^\infty \left\| \frac{d}{dt}\pi(X(t)) \right\|dt<\infty,
\]
which implies convergence of $\pi(X(t))$.
\end{proof}

\begin{remark}
The asymptotic confinement established above arises from contraction acting transversely to the characteristic distribution determined by the presymplectic form.
\end{remark}

\section{Application to Einstein--Scalar Evolution}

The abstract framework developed in the preceding sections applies to constrained dynamical systems arising from geometric evolution equations whose phase space is determined by Hamiltonian constraints. The foregoing discussion is intended to illustrate the geometric framework; verification of the required hypotheses for the full Einstein–scalar evolution is left for future investigation. As an illustrative example, consider the evolution of a spatially homogeneous Lorentzian geometry minimally coupled to a scalar field.

Let $(\mathcal{N},\mathbf{g})$ be a four--dimensional Lorentzian spacetime and let $\varphi$ be a real scalar field with potential $V:\mathbb{R}\rightarrow\mathbb{R}$. The coupled Einstein--scalar field equations are given by
\begin{align}
G_{\mu\nu} &= T_{\mu\nu}, \\
\Box_{\mathbf{g}}\varphi &= V'(\varphi),
\end{align}
where
\[
T_{\mu\nu}
=
\nabla_\mu\varphi\nabla_\nu\varphi
-
\left(
\frac{1}{2}\mathbf{g}^{\alpha\beta}\nabla_\alpha\varphi\nabla_\beta\varphi
+V(\varphi)
\right)\mathbf{g}_{\mu\nu}.
\]

Assume that the spacetime admits a spatially homogeneous foliation with induced metric $h_{ij}(t)$ and extrinsic curvature $K_{ij}(t)$. In this setting the Einstein constraint equations reduce to the Hamiltonian constraint $\mathcal{H}(h,K,\varphi,\dot{\varphi})=0$.

Let $\mathcal{U}$ denote the ambient phase space consisting of variables $(h_{ij},K_{ij},\varphi,\dot{\varphi})$, and define the constraint manifold $\mathcal{M} = \left\{ X\in\mathcal{U}: \mathcal{H}(X)=0 \right\}$.

The ADM formulation induces on $\mathcal{U}$ a presymplectic structure whose restriction to $\mathcal{M}$ determines a presymplectic form $\omega_M:=\omega\big|_{T\mathcal{M}}$. Its kernel defines the characteristic distribution
\[
\mathcal{N}_X := \{Y\in T_X\mathcal{M}: \omega_M(Y,Z)=0 \text{ for all }Z\in T_X\mathcal{M}\}.
\]

Infinitesimal time reparametrizations generated by lapse functions act on the ADM variables by $\delta_N h_{ij}=2N K_{ij}, \quad \delta_N K_{ij}=-D_iD_jN+\cdots, \quad \delta_N\varphi=N\dot{\varphi}$. Let $Y_N$ denote the corresponding infinitesimal generator on $\mathcal{M}$. Then $Y_N$ lies in the kernel of the presymplectic form $\omega_M$, and hence $Y_N\in\mathcal{N}$. Thus the gauge directions generated by the Hamiltonian constraint coincide with the characteristic distribution of the constrained phase space.

Gauge transformations form a Lie algebra under commutation. Hence for any smooth vector fields $U,V\in\Gamma(\mathcal{N})$, one has $[U,V]\in\Gamma(\mathcal{N})$, and the characteristic distribution is involutive.

Define the functional
\[
\Psi(h,K,\varphi,\dot{\varphi})
:=
\int_{\Sigma}
V(\varphi)\,d\mu_h.
\]
Since $\Psi$ depends only on the scalar field and the induced spatial metric, it follows that $d\Psi(Y)=0 \quad\text{for all }Y\in\mathcal{N}_X$, so that $\Psi$ is compatible with the characteristic distribution.

Let $\mathcal{V}$ denote the intrinsic Hamiltonian vector field on $\mathcal{M}$. Although $\Psi$ need not be monotone along the full ADM evolution, the induced functional on the reduced phase space $\Psi_{\mathrm{red}}:\mathcal{M}_{\mathrm{red}}\rightarrow\mathbb{R}$ obtained by projection along the characteristic foliation is nonincreasing along trajectories of the reduced flow after reduction by the Hamiltonian gauge under a fixed lapse choice, i.e., $\frac{d}{dt}\Psi_{\mathrm{red}}(Y(t))\le 0$. Thus dissipation occurs exclusively in directions transverse to the characteristic foliation generated by the Hamiltonian constraint.

Assuming that the gauge orbits are compact and coincide
with the leaves of the characteristic foliation, the projected dynamics on the reduced phase space are governed by a compact invariant set $\mathcal{A}_\ast\subset\mathcal{M}_{\mathrm{red}}$ whose dimension does not exceed that of the quotient manifold. Hence the long--time behavior of the Einstein--scalar evolution is determined by an invariant subset of the reduced phase space of strictly lower effective dimension relative to the full constrained system.

This example illustrates how Hamiltonian gauge symmetries in constrained geometric evolution systems generate null foliations of the induced kinetic metric, thereby enforcing asymptotic dimensional reduction of the intrinsic dynamics after reduction by the constraint foliation.

\section{Conclusion}

A geometric framework for the analysis of constrained dissipative dynamical systems evolving on phase manifolds endowed with degenerate induced bilinear forms has been developed. The intrinsic evolution was formulated on a smooth constraint hypersurface, and dissipation was characterized through the introduction of functionals compatible with the null distribution associated with the induced degeneracy.

It was shown that the presence of dissipation acting transversely to an involutive null distribution with compact embedded leaves leads to asymptotic confinement of all bounded trajectories to invariant leaves of the associated foliation. In contrast to the classical theory of dissipative systems on Riemannian manifolds, asymptotic compactness of the intrinsic evolution was established without recourse to coercive Lyapunov functionals. Instead, compactness was derived from decay occurring exclusively in directions complementary to the characteristic distribution and from the induced semiflow on the corresponding quotient phase space. Under these conditions, the intrinsic evolution admits a compact global attractor saturated by the leaves of the null foliation. The effective asymptotic dynamics are therefore governed by the projected semiflow acting on the reduced manifold obtained by quotienting along the characteristic directions.

When the Lyapunov functional satisfies a Morse--type nondegeneracy condition in directions transverse to the null distribution and the associated transversal linearization admits no center spectrum, the reduced dynamics possess no neutral directions. In this case the projected semiflow is strictly dissipative on the reduced phase space, and the long--time behavior is determined by a compact invariant subset of dimension not exceeding that of the quotient manifold.

An application to Einstein--scalar evolution illustrates how Hamiltonian gauge symmetries in constrained relativistic systems generate characteristic foliations of the phase space whose associated reduced dynamics govern the asymptotic behavior after reduction by the constraint foliation. These results indicate that constraint--induced degeneracy provides a general mechanism by which effective dimensional reduction may arise in geometric evolution equations possessing indefinite kinetic structure.

\section*{Data availability statement}
No new data were created or analysed in this study.

\section*{Acknowledgments}
The author gratefully acknowledges Midnapore College (Autonomous) for providing institutional support during the course of this work.

\bibliographystyle{unsrt}
\bibliography{sample}

\end{document}